\documentclass[11pt]{article}

\usepackage[utf8]{inputenc}
\usepackage{fullpage}
\usepackage{amssymb,amsthm,amsmath,amsfonts}
\usepackage{pst-all,pst-3dplot,pstricks,pstricks-add,pst-math,pst-xkey}
\usepackage{graphicx}
\usepackage{enumerate}
\usepackage{cleveref}

\newtheorem{theorem}{Theorem}
\newtheorem{proposition}[theorem]{Proposition}
\newtheorem{corollary}[theorem]{Corollary}
\newtheorem{lemma}[theorem]{Lemma}
\newtheorem{claim}[theorem]{Claim}
\newtheorem{definition}[theorem]{Definition}

\theoremstyle{definition}
\newtheorem{example}[theorem]{Example}

\theoremstyle{remark}
\newtheorem{remark}[theorem]{Remark}

\newcommand{\N}{\mathcal N}
\renewcommand{\P}{\mathcal P}
\renewcommand{\L}{\mathcal L}
\newcommand{\SD}{\mathcal{SD}}
\newcommand{\WD}{\mathcal{WD}}
\newcommand{\F}{\mathcal{F}}
\newcommand{\UI}{\mathcal{UI}}
\newcommand{\R}{\mathcal R}
\renewcommand{\sl}{S_{\mathcal L}}
\newcommand{\sr}{S_{\mathcal R}}

\DeclareMathOperator{\dist}{\mathbf{d}}

\title{Partizan Subtraction Games\thanks{Supported by the ANR-14-CE25-0006 project of the French National Research Agency}}
\author{Eric Duch\^ene$^1$, Marc Heinrich$^1$, Richard J. Nowakowski$^2$, Aline Parreau$^1$}

\date{$^1$Univ Lyon, Université Lyon 1, LIRIS UMR CNRS 5205, F-69621, Lyon, France.\\ $^2$Dalhousie University, Halifax, Canada.}

\begin{document}

\maketitle


\begin{abstract}
Partizan subtraction games are combinatorial games where two players, say  Left  and  Right, alternately remove a number $n$ of tokens from a heap of tokens, with $n\in\sl$ (resp. $n\in\sr$) when it is  Left's (resp.  Right's) turn. The first player unable to move loses. These games were introduced by Fraenkel and Kotzig in 1987, where they introduced the notion of dominance, i.e. an asymptotic behavior of the outcome sequence where  Left  always wins if the heap is sufficiently large. In the current paper, we investigate the other kinds of behaviors for the outcome sequence. In addition to dominance, three other disjoint behaviors are defined, namely {\em weak dominance}, {\em fairness} and {\em ultimate impartiality}. We consider the problem of computing this behavior with respect to $\sl$ and $\sr$, which is connected to the well-known Frobenius coin problem. General results are given, together with arithmetic and geometric characterizations when the sets $\sl$ and $\sr$ have size at most $2$.
\end{abstract}

\section{Introduction}

Partizan subtraction games were introduced by Fraenkel and Kotzig in 1987 \cite{FrKo87}. They are 2-player combinatorial games played on a heap of tokens. Each player is assigned a finite set of integers, respectively denoted $\sl$ (for the  Left  player), and $\sr$ (for the  Right  player). A move consists in removing a number $m$ of tokens from the heap, provided $m$ belongs to the set of the player. The first player unable to move loses. When $\sl=\sr$, the game is impartial and known as the standard {\sc subtraction game}---see \cite{WW}.

We now recall the useful notations and definitions coming from combinatorial game theory. More information can be found in the reference book \cite{Siegel}. There are two basic \textit{outcome} functions: for a position $g$,
\[
o_L(g)=\begin{cases}
   L, \text{  if Left moving first has a winning strategy;}\\
   R, \text{ otherwise;}
\end{cases}
\]
and
\[
o_R(g)=\begin{cases}
   R, \text{  if Right moving first has a winning strategy;}\\
   L, \text{ otherwise.}
\end{cases}
\]

It is usual to talk of the  {\em outcome} of a position $g$ and the associated outcome function $o(g)$,
\begin{itemize}
\item For $o_L(g)=o_R(g)=L$---Left wins regardless of who moves first, written $o(g)=\L$;
\item For $o_L(g)=o_R(g)=R$---Right wins regardless of who moves first, written $o(g)=\R$;
\item For $o_L(g)=L$, $o_R(g)=R$---the player who starts has a winning strategy, $o(g)=\N$;
\item For $o_L(g)=R$, $o_R(g)=L$---the second player has a winning strategy, $o(g)=\P$.
\end{itemize}
In outcome function, there should be a reference to the game/rules. In this paper, the position will be a number but the rules will be clear from the context so the rules will not be included in the function.

A partizan subtraction game $G$ with rules $(\sl,\sr)$ will be denoted $(\sl, \sr)$ in the rest of the paper. A game position of $G$ will be simply denoted by an integer $n$ corresponding to the size of the heap. The {\em outcome sequence} of $G$ is the sequence of the outcomes for $n=0,1,2,3,\ldots$, i.e., $o(0), o(1), o(2), \ldots$.
 A well-known result ensures that the outcome sequence of any impartial subtraction game is ultimately periodic~\cite{Siegel}. Note that in that case, the outcomes only have the values $\P$ and $\N$ since the game is impartial. In \cite{FrKo87}, this result is extended to partizan subtraction games.

\begin{theorem}[Fraenkel and Kotzig \cite{FrKo87}]
The outcome sequence of any partizan subtraction game is ultimately periodic.
\end{theorem}

\begin{example}\label{ex:1213}
Consider the partizan subtraction game $G=(\{1,2\},\{1,3\})$. The outcome sequence of $G$ is
$$
\P\; \N\; \L\; \N\; \L\; \L\; \L\; \L\; \cdots
$$
In this particular case, the periodicity of the sequence can be easily proved by showing by induction that the outcome is $\L$ for $n\geq 4$.
\end{example}

Such a behavior where the outcome sequence has period $1$ is rather frequent for partizan subtraction games. In that case, the period is only $\L$ or $\R$. In their paper, Fraenkel and Kotzig called this property {\em dominance}. More precisely, we say that $\sl\succ \sr$ - or that $\sl$ dominates $\sr$ - if there exists an integer $n_0$ such that the outcome of the game $(\sl, \sr)$ is always $\L$ for all $n\geq n_0$. By symmetry, a game satisfying $\sl\prec \sr$ is always $\R$ for all sufficiently large heap sizes. When a game neither satisfies $\sl\succ \sr$ nor $\sl\prec \sr$, the sets $\sl$ and $\sr$ are said {\em incomparable}, denoted by $\sl \| \sr$. In \cite{FrKo87}, several instances have been proved to satisfy the dominance property (i.e. the games $(\{1,2m\},\{1,2n+1\})$ and $(\{1,2m\},\{1,2n\})$), or to be incomparable like $(\{a\},\{b\})$. It is also shown that the dominance relation is not transitive. Note that in \cite{Plam95}, the game values (i.e. a refinement of the outcome notion) have been computed for the games $(\{1,2\},\{1,k\})$.

In the literature, partizan taking and breaking games have not been so much considered. A more general version, where it is also allowed to split the heap into two heaps, was introduced by Fraenkel and Kotzig in \cite{FrKo87}, and is known as partizan octal games. A particular case of such games, called {\em partizan splittles}, was considered in \cite{Mesdal}, where, in addition, $\sl$ are $\sr$ are allowed to be infinite sets. Another variation with infinite sets is when $\sl$ and $\sr$ make a partition of $\mathbb{N}$ \cite{WPS}. In such cases, the ultimate periodicity of the outcome sequence is not necessarily preserved.

In the current paper, we propose a refinement of the structure of the outcome sequence for partizan subtraction games. More precisely, when the sets $\sl$ and $\sr$ are incomparable, different kinds of periodicity can occur. The following definition presents a classification for them.

\begin{definition}\label{def:sequence}
The outcome sequence of $G = \textsc{Subtraction}(\sl, \sr)$ is:
\begin{itemize}
\item $\SD$ (Strongly Dominating) for  Left  (resp.  Right ), and we write $\sl \succ \sr$ (resp. $\sl \prec \sr$) if any position $n$ large enough has outcome $\L$ (resp. $\R$). In other words, the period is reduced to $\L$ (resp. $\R$).
\item $\WD$ (Weakly Dominating) for  Left  (resp.  Right ), and we write $\sl >_w \sr$ if the period contains at least one $\L$ and no $\R$ (or resp. one $\R$ and no $\L$).
\item $\F$ (Fair) if the period contains both $\L$ and $\R$.
\item $\UI$ (Ultimately Impartial) if the period contains no $\L$ and no~$\R$.
\end{itemize}
\end{definition}

\begin{remark}\label{rem:seq_interdites}
Note that inside a period, not all the combinations of $\P$, $\N$, $\L$ and $\R$ are possible. For example, in
 a game that is not $\UI$, a period that includes $\P$ must include $\N$. Indeed, assume on the contrary it is
 not the case and let $n$ be a position of outcome $\P$ in the period, where the period has length $p$.
  Let $a\in\sl$. Now the position $n+a$ is in the period, and $o(n+a)= \L$ since Left can win going first and,
  by assumption, $o(n+a)\ne \N$. For the same reason, $o(n+2a)= \L$. By repeating this argument,  $o(n+ka)= \L$ for all $k$.
  Since $n$ is in the period, we now have $\P=o(n)=o(n+pa)=\L$,  a contradiction.
  \end{remark}

If the literature detailed above give examples of $\SD$ and $\UI$ games (as impartial subtraction games are $\UI$), we will see later in this paper examples of $\WD$ games (e.g. Lemma~\ref{lem:1kk}). The example below shows an example of a fair game.
\begin{example}\label{ex:intro}
    Let $\sl = \{ c,c+1 \}$ and $\sr = \{ 1, b\}$ with $b = c(c+1)$ and $c>1$. Then the game $(\sl, \sr)$ is $\F$.
\end{example}
\begin{proof}
We proceed by induction on the size of the heap, in order to show there are infinitely many $\L$ and $\R$. Since $c>1$, we have $o(1)=\R$, and  $o(c+1)=\L$. Now we assume that for some $n$, $o(n)=\L$, and  show that $o(n+b+c)=\L$. In the position $n+b+c$, Left considers these as the
two heaps $n$ and $b+c$, and if Right removes $1$, Left regards this as a move in the $n$ component else it is a move in the second heap.
Left moving first applies her winning strategy on $n$ and then,
regardless of whether Left moved first or second, responds in the remnants of $n$ heap whenever Right removes $1$ token. If at some point, Right chooses to remove $b$ tokens, then Left answers immediately by removing $c$ tokens, eliminating the second heap. In that case, Left wins at the end by applying her winning strategy on $n$. On the contrary, if Right never plays $b$, then Left empties the $n$ component and it is Right's turn from the $b+c$ position. Again, from $b+c$, playing $b$ is a losing move for Right. If he plays $1$, then Left plays $c+1$, leading to the position $b-2=(c-1)(c+2)$. All the next legal moves of Right are $1$, and all the answers of Left are $c+1$, which guarantees to empty the position and hence win the game.\\

Assume now that $o(n)=\R$ and we show that $o(n+b+c)=\R$. As previously, Right considers this position as the two heaps $n$ and $b+c$.
He applies his winning strategy on $n$ and any move $c$ of Left leads Right to answer by removing $b$ tokens, leaving a winning position for  Right. Hence assume that  Left  plays $c+1$ until  Right  wins on $n$. At this point,  Left  has to play from a position $k+b+c$ with $k<c$. If $k=0$, then  Left  loses for the same reasons as in the above case (as the position $b+c$ is $\P$). Otherwise, any move $c$ or $c+1$ of  Left  is followed by a move $b$ of  Right, leading to a position with at most $k$ tokens, from which  Left  cannot play and loses.
\end{proof}

The paper is organized as follows. In Section 2, we consider the two decision problems related to the computation of the outcome of a game position and of the behavior of the outcome sequence. Links with the Frobenius coin problem and the knapsack problem are given. Then, we try to characterize the behavior of the outcome sequence ($\SD$, $\WD$, $\F$ or $\UI$) according to $\sl$ and $\sr$. When $\sl$ is fixed, Section 3 gives general results about strong and weak dominance according to the size of $\sr$. In Section 4 and 5, we characterize the behavior of the outcome sequence when $|\sr|=1$ and $|\sl|\leq 2$. Section 6 is devoted to the case $|\sl|=|\sr|=2$, where it is proved that the sequence is mostly strongly dominating.

\section{Complexity}\label{sec:complexity}

Computing the outcome of a game position is a natural question when studying combinatorial games. For partizan subtraction games, we know that the outcome sequence is eventually periodic. This implies that, if $\sl$ and $\sr$ are fixed, computing the outcome of a given position $n$ can be done in polynomial time. However, if the subtraction sets are part of the input, then the algorithmic complexity of the problem is not so clear. This problem can be expressed as follows:\\

\noindent {\sc psg outcome} \\
\underline{Input:} two sets of integers $\sl$ and $\sr$, a game position $n$ \\
\underline{Output:} the outcome of $n$ for the game $(\sl,\sr)$\\

In the next result, we show that this problem is actually NP-hard.

\begin{theorem}
\label{thm:ukp}
\textsc{psg outcome} is NP-hard, even in the case where the set of one of the players is reduced to one element.
\end{theorem}
\begin{proof}
We use a reduction to {\sc Unbounded Knapsack Problem} defined as follows. \\

\noindent {\sc Unbounded Knapsack Problem} \\
\underline{Input:} a set $S$ and an integer $n$ \\
\underline{Output:} can $n$ be written as a sum of non-negative multiples of $S$?\\

{\sc Unbounded Knapsack Problem} was shown to be NP-complete in~\cite{knapsack}.\\

Let $(S,n)$ be an instance of {\sc unbounded knapsack problem}, where $S$ is a finite set of integers, and $n$ is a positive integer. Without loss of generality, we can assume that $1 \not \in S$ since otherwise the problem is trivial. We consider the partizan subtraction game where  Left  can only play $1$, and  Right  can play any number $x$ such that $x + 1 \in S$. In other words, we have $S_L = \{ 1\}$ and $S_R = S -1$. We claim that for this game,  Right  has a winning strategy playing second if and only if $n$ can be written as a sum of non-negative multiples of elements of $S$.

    Observe that during one round (i.e. one move of  Left  followed by one move of  Right), if $x$ is the number of tokens that were removed, then $x \in S$. Suppose that Right has a winning strategy, and consider any play where  Right  plays according to this strategy. Then  Right  makes the last move, and after this move no token remains. Indeed, if there was at least one token remaining, then  Left  could still remove this token and continue the game. At each round an element of $S$ was removed, and at the end, no tokens remains. This implies that $n$ is a sum of non-negative multiples of $S$.

    In the other direction, if $n$ is a sum of non-negative multiples of $S$, we can write $n = \sum_{x \in S} n_x x$. A winning strategy for  Right  is simply to play $n_x$ times the move $x-1$ for each $x \in S$.
\end{proof}

\begin{remark}
In the case of impartial subtraction games (i.e. $\sl=\sr$), there is no known result about the complexity of this problem. This is surprising as these games have been thoroughly investigated in the literature.
\end{remark}

The second question that emerged from partizan subtraction games is the behavior of the outcome sequence, according to Definition~\ref{def:sequence}. It can also be formulated as a decision problem.\\

\noindent {\sc psg sequence} \\
\underline{Input:} two sets of integers $\sl$ and $\sr$\\
\underline{Output:} is the game $(\sl,\sr)$ $\SD$, $\WD$ (and not $\SD$), $\F$ or $\UI$ ?\\

Unlike {\sc psg outcome}, the algorithmic complexity is open for {\sc psg sequence}. The next sections will consider this problem for some particular cases. In addition, one can wonder whether the knowledge of the sequence could help to compute the outcome of a game position. The answer is no, even if the game is $\SD$:

\begin{proposition}\label{prop:preperiod}
Let $\sl=\{a_1,\ldots,a_n\}$ be such that $\gcd(a_1+1,\ldots,a_n+1)=1$, and let $\sr=\{1\}$. The game $(\sl,\sr)$ is $\SD$ for  Left  but computing the length of the preperiod is NP-hard.
\end{proposition}

The proof will be based on the well-known {\sc coin problem} (also called Frobenius problem). \\

\noindent {\sc coin problem} \\
\underline{Input:} a set of $n$ positive integers $a_1,\ldots , a_n$ such that $\gcd(a_1,\ldots,a_n)=1$\\
\underline{Output:} the largest integer that cannot be expressed as a linear combination of $a_1,\ldots,a_n$.\\

This value is called the {\em Frobenius number}. For $n=2$, the Frobenius number equals $a_1a_2-a_1-a_2$ ~\cite{Sylvester}\footnote{Although not germane to this paper, Sylvester's solution is central to the strategy stealing argument that proves that naming a prime 5 or greater is a winning move in \textsc{sylver coinage}\cite{WW}, chapter 18.}. No explicit formula is known for larger values of $n$. Moreover, the problem has been proved to be NP-hard in the general case \cite{alfonsin}.

\begin{proof}
Under the assumptions of the proposition, we will show that the length of the preperiod is exactly the Frobenius number of $\{a_1+1,\ldots,a_n+1\}$. Indeed, let $N$ be the Frobenius number of $\{a_1+1,\ldots,a_n+1\}$. Then $N+1, N+2\ldots$ can be written as a linear combinations of $\{a_1+1,\ldots,a_n+1\}$. Note that in the game $(\sl,\sr)$, any round (sequence of two moves) can be seen as a linear combination of $\{a_1+1,\ldots,a_n+1\}$, as  Left  plays an $a_i$ and  Right  plays $1$. Hence if  Right  starts from $N+1$,  Left  follows the linear combination for $N+1$ to choose her moves, so as to play an even number of moves until the heap is empty. For the same reasons, if  Right  starts from $N+2$,  Left  has a winning strategy as a second player. Since  Right's first move is necessarily $1$, it means that  Left  has a winning strategy as a first player from $N+1$. Thus the position satisfies $o(N+1)=\L$. Using the same arguments, this remains true for all positions greater than $N+1$. In other words, it proves that the game is $\SD$ for  Left. Now, we consider the position $N$ and show that $o(N)\ne\L$. Indeed, assume that  Right  starts and  Left  has a winning strategy. It means that an even number of moves will be played. According to the previous remark, the sequence of moves that is winning for  Left  is necessarily a linear combination of $\{a_1+1,\ldots,a_n+1\}$. This contradicts the Frobenius property of $N$.
\end{proof}

This correlation between partizan subtraction games and the coin problem will be reused further in this paper.

\section{When $\sl$ is fixed}

In this section, we consider the case where $\sl$ is fixed and study the behaviour of the sequence when $\sr$ varies. In particular, we look for sets $\sr$ that make the game $(\sl,\sr)$ favorable for  Right. This can be seen as a prelude to the game where players would choose their sets before playing: if  Left  has chosen her set $\sl$, can  Right  force the game to be asymptotically more favorable for him?

\subsection{The case $|\sr|>|\sl|$}

If $\sr$ can be larger than $\sl$, then it is always possible to obtain a game favorable for  Right, as it is proved in the following theorem.

\begin{theorem}
Let $\sl$ be any finite set of integers. Let $p$ be the period of the impartial subtraction game played with $\sl$ and let $\sr=\sl\cup \{p\}$.
Then  Right  strongly dominates the game $(\sl,\sr)$, i.e., the game $(\sl,\sr)$ is ultimately $\R$.
\end{theorem}

\begin{proof}
Let $n_0$ be the preperiod of the impartial subtraction game played on $\sl$ and $m$ be the maximal value of $\sl$. We prove that  Right  wins if he starts on any heap of size $n> n_0+p$, which implies that the outcome on  $(\sl,\sr)$ is $\R$ for any heap of size $n> n_0+p+m$.

If $n$ is a $\N$-position for the impartial subtraction game on $\sl$, then  Right  follows the strategy for the first player, never uses the value $p$, and wins.

If $n$ is a $\P$-position,  Right  takes $p$ tokens, leaving  Left  with a heap of size $n-p> n_0$ which is, using periodicity, also a $\P$-position in the impartial game. After Left's move, we are in the case of the previous paragraph and Right wins.\end{proof}

Note that in the previous theorem, $\sr$ contains the set $\sl$, and thus has a large common intersection. We prove in the next theorem that if $\sr$ cannot contain any value in $\sl$, then it is still possible to have a game that is at least fair for  Right  (i.e., it contains an infinite number of $\R$-positions). Note that we do not know if for any set $\sl$, there is always a set $\sr$ with $|\sr| = |\sl| +1$ and $\sr\cap \sl =\emptyset$ that is (weakly or strongly) dominating for  Right.

\begin{theorem}
For any set $\sl$, there exists a set $\sr$ with $\sl\cap \sr=\emptyset$ and $|\sr| = |\sl| +1$ such that the resulting game contains an infinite number of $\R$-positions.
\end{theorem}

\begin{proof}
Let $n$ be any integer such that the set $A=\{n-m,m\in \sl\}$ is a set of positive integers that is disjoint
from $\sl$. Putting $\sr = A \cup \{ n\}$ gives a set that satisfies the condition of the theorem and the game
 $(\sl,\sr)$.

 We claim that $o(kn)=\R$ for $k=1,2,\ldots$.
If  Left  starts on a position $kn$ by removing $m$ tokens, then  Right  can answer by taking $n-m$ tokens and leaves  $(k-1)n$ tokens, and by induction, Right wins.
If  Right  starts, he takes $n$ tokens and, again,  Left  has a multiple of $n$ and  loses.
\end{proof}

Consequently, if  Right  has a small advantage on the size of the set, he can ensure that the sequence of outcomes contains an infinite number of $\R$-positions. So having a larger subtraction set seems to be an important advantage. However, having a larger set is not always enough to guarantee
 dominance. Indeed, we have the following result:

\begin{theorem}
	\label{th+perstay}
	Let $G = (S_L, S_R)$ be a partizan subtraction game. Assume that $|S_L|\geq 2$ and that $G$ is eventually $L$, with preperiod at most $p$. Let $x_1, x_2 \in S_L$, with $x_1 < x_2$, and let $d$ be an integer with $d> p + \max(S_R \cup \{x_2 - x_1\})$, then $G_d = (S_L, S_R \cup \{d\})$ is eventually $L$ with preperiod at most $(d + x_2) \lceil\frac {d+x_2}{x_2 - x_1} \rceil$
\end{theorem}

\begin{proof}
	Let $G$, $d$, $x_1$ and $x_2$ be as in the statement of the theorem.
	We first prove the following claim:
	
	\begin{claim}
			In the game $G_d$,
			if $o_L(n) = \L$  (resp. $o_R(n) = \L$) then Left  has a winning strategy on
			$n + (d+x)$ as first (resp. second player), with $x \in S_L$.
	\end{claim}
	\begin{proof}
		We will show the result by induction on $n$.
		
		First, assume $o_R(n) = \L$. We will show that there is a winning strategy for  Left  playing second on $n + d + x$. Starting from the position $n+d+x$, there are three possible cases:
		\begin{itemize}
			\item  Right  plays $y \in S_R$, with $y \leq n$. By the assumption on $n$,
Left  wins as first player on $n - y$, and using the induction hypothesis, he also wins as first player on
$n - y + d + x$. Therefore, Left wins as second player on $n+d+x$.
			\item  Right  plays $y \in S_R$, with $y > n$. Now Left  answers by playing $x$. This leads to the position $(n-y) + d$, with $(n-y) + d > p$ by assumption on $d$, and $n-y + d < d$ by assumption on $y$. Since $n-y + d < d$,  Right  can no longer play his move $d$, and the outcome of $G_d$ on $n-y+d$ is the same as the outcome of $G$ on this position. Since $n-y+d > p$  Left  wins playing second on this position.
			\item  Right  plays $d$, then  Left  answers by playing $x$, leading to the position $n$ on which  Left  wins as second player by assumption.
		\end{itemize}
		
		Suppose now that  Left  wins playing first on $n$, and let $y \in S_L$ be a winning move for  Left. Then  Left  wins playing second on $n-y$, and using the induction hypothesis, she wins playing second on $n-y + d + x$. Consequently, $y$ is a winning move for  Left  on $n + d + x$.
	\end{proof}	
	 For $i \geq 0$, denote by $X_i$ the set of integers $k < d +x_2$ such that the position $i (d + x_2) + k$ is $\L$ for $G_d$. To prove the theorem, it is enough to show that if $i$ is large enough, then $X_i = [0, x_2 + d [$. From the claim above, we know that $X_i \subseteq X_{i+1}$.
	
	Additionally, using the hypothesis on $d$, we have that $[p+1, d-1] \subseteq X_0$. Finally, we have the following property. For any $x \geq 0$, if $x \in X_i$, then $x - (x_2 - x_1) \mod (d+x_2) \in X_{i+1}$. Indeed, if $x \in X_i$, then $i (d+x_2) + x$ is an $\L$-position, and using the claim above, so is $i (d + x_2) + x + d + x_1 = (i+1) (d+x_2) + x -  (x_2 - x_1)$.
	
	Let $0 \leq x < d + x_2$, and write $(d - x) \mod (d + x_2) = \alpha (x_2 - x_1) + \beta$ the euclidian division of $(d-x) \mod (d + x_2)$ by $(x_2 - x_1)$. We have $0 < \beta \leq x_2 - x_1$, and $\alpha \leq \lceil \frac {d+x_2}{x_2 - x_1} \rceil$. This can be rewritten as:
	
	$$ x = (d - \beta) - \alpha (x_2 - x_1) \mod (d + x_2)$$
	
	Since we know that $d - \beta \geq p$ by assumption on $d$, we have that $(d - \beta) \in X_0$, and using the observation above, this implies that $x \in X_{\alpha} \subseteq X_{\lceil \frac {d+x_2}{x_2 - x_1} \rceil}$.
	
	Consequently, $G_d$ is ultimately $\L$, and the preperiod is at most $(d + x_2) \lceil\frac {d+x_2}{x_2 - x_1} \rceil$.
\end{proof}

By applying iteratively Theorem~\ref{th+perstay} with a game that is $\SD$ for  Left  (like the game of Example~\ref{ex:1213}), we obtain the following corollary.

\begin{corollary}
There are sets $\sl$ and $\sr$ with $|\sl| = 2$ and $|\sr|$ arbitrarily large such that $(\sl,\sr)$ is $\SD$ for  Left .
\end{corollary}

\begin{remark}
    The condition on $d \geq p + \max(S_R \cup \{x_2 - x_1\})$ in Theorem~\ref{th+perstay} is optimal.
    Indeed, take $\sl=\{c,c+1\}$ and $\sr=\{1\}$. As seen in the proof of Proposition \ref{prop:preperiod}, the game $(\sl,\sr)$ is $\SD$ for  Left, with preperiod the Froebenius number of $\{c+1,c+2\}$, which is $p=c^2+2c-1=c(c+1)-1$.
    Thus, by Theorem~\ref{th+perstay}, the game $(\{c,c+1\},\{1,d\})$ with $d>c(c+1)$ is also $\SD$ for  Left. But, as proved in Example \ref{ex:intro}, this is not true for $d=c(c+1)$ since the game is then $\F$.
\end{remark}

\subsection{The case $|\sr|\leq |\sl|$}

We first consider the case $\sl=\{1,\ldots,k\}$ and prove that the game is always favorable to Left
and that $\sl$ strongly dominates in all but a few cases.

\begin{lemma}\label{lem:1kk}
Let $\sl = \{1, \ldots, k\}$, and $|\sr| = k$, then:
\begin{enumerate}
\item If $\sr = \{ c+1, c+2, \ldots c+k\}$ for some integer $c$, then  Left  weakly dominates if $c>0$ and the game is impartial if $c=0$,
\item otherwise,  Left  strongly dominates.
\end{enumerate}
\end{lemma}

\begin{proof}
\begin{enumerate}
\item In this case, the game is purely periodic, with period $\P \L^{c} \N^{k}$. This can be proved by induction on the size of the heap $n$. If $0<n\leq c$, only  Left  can play and the game is trivially $\L$.
Otherwise, let $x=n\bmod {c+k+1}$. If $x=0$, then if the first player removes $i$ tokens, the second player answers by removing $c+k+1-i$ tokens, leading to the position $n-c-k-1$ which is $\P$ by induction, and so is $n$.
If $0<x<c+1$, when  Left  starts she takes one token, leading to a $\L$ or a $\P$-position, and wins. If she is second, she plays as before to $n-c-k-1$ which is a $\L$-position.
Finally, if $x\geq c+1$, both players win playing first by playing $x-c$ for  Left  and $x$ for  Right.

\item We show that if  $n > 0$ is such that  Right  wins playing second on $n$, this implies that $\sr$ contains $k$ consecutive integers. Let $n_0$ be the smallest positive integer such that $o_L(n_o)=R$.
 We know that $n_0 > k$ since otherwise  Left  can win playing first by playing to zero. Since  Right  has a winning strategy playing second then  Right  has a winning first move on all the position $n-i$ for $1 \leq i \leq k$. This means that for each of these positions,  Right  has a winning move to some position $m_i$ where $o_L(m_i)=R$. By minimality of $n_0$, this implies that $m_i =0$, and consequently $n-i \in \sr$ for all $1 \leq i \leq k$.
Consequently, if $\sr$ does not contain $k$ consecutive integers, there is no position $n>0$ such that  Right  wins playing second. In particular, there is no $\R$ nor $\P$-positions in the period. By Remark \ref{rem:seq_interdites}, this implies that the period only contains $\L$-positions, meaning that the game is strongly dominating for  Left.
\end{enumerate}
\end{proof}

The set $\sl = \{ 1, \ldots, k\}$ is somehow optimal for  Left, since the exceptions of strongly domination for  Left  in the previous lemma appear for any set of $k$ elements:

\begin{lemma}
    For any set $\sl$, there is a set $\sr$ with $|\sr| = |\sl|$ and $\sr\cap\sl=\emptyset$ such that  Left  does not strongly dominate.
\end{lemma}
\begin{proof}
Let $\sr=n_0-\sl$ for an integer $n_0$ larger than all the values of $\sl$ and such that $\sr\cap \sl =\emptyset$. Then  Right  wins playing second in all the multiples of $n_0$.
\end{proof}

\section{When one set has size 1}

We now consider the case where one of the set, say $\sr$ has size 1. As seen in Section \ref{sec:complexity}, the study of the game is closely related to {\sc Unbounded Knapsack Problem} and to the coin problem. Indeed,  Right  does not have any choice and thus the result is only depending on the possibility or not for $n$ to be decomposed as a combination of the values in $\sl+\sr$. Our aim in this section is to exhibit the precise periods.

\subsection{Case $|\sl|=|\sr|=1$}

In this really particular case, the game is always $\WD$ for the player that have the smallest integer.

\begin{lemma}\label{lem:one_v_one}
Let $S_L=\{a\}$ and $S_R=\{b\}$ with $a< b$. The outcome sequence of $S=(S_L,S_R)$ is
purely periodic, the period length is $a+b$ and the period is $\P^a\L^{b-a}\N^a$. In particular, the game is weakly dominating for  Left.
\end{lemma}
\begin{proof}

We prove that for all $n \geq 0$, if one of the player has a winning move playing first (resp. second) on $n$, then he also has one playing first (resp. second) on $n+a+b$. Indeed, suppose for example that  Left  has a winning move on position $n$ playing first (the other cases are treated in the same way). If  Left  plays first on position $n+a+b$, then after two moves, it's again  Left 's turn to play, and the position is now $n$, and  Left  wins the game.

The result then follows from computing the outcome of the positions $n \leq a+b$. These outcomes are tabulated in \Cref{tab+1vs1}.

\begin{table}[htp]
\begin{center}
\begin{tabular}{c|c|c|c}
Heap sizes      &    Left  move range &   Right  move range    &   Outcome\\ \hline
$[0,a-1]$       &    no moves       &    no moves           &    $\P$\\
$[a,b-1]$       &    $[0,b-a-1]$      &    no moves           &    $\L$\\
$[b,b+a-1]$     &   $[b-a,b-1]$     &   $[0,a-1]$           &   $\N$\\
$[b+a, b+2a-1]$ &   $[b,b+a-1]$     &   $[a,2a-1]$          &   $\P$\\
$[b+2a,2b+2a-1]$&   $[b+a,2b+a-1]$  &   $[2a,b+2a-1]$       &   $\L$\\
\end{tabular}
\end{center}
\caption{Outcomes with $S_L=\{a\}$ and $S_R=\{b\}$ for first values}
\label{tab+1vs1}
\end{table}
\end{proof}


\subsection{Case $|\sl|=2$ and $|\sr|=1$}

In these cases, we are able to give the complete periods.

\begin{theorem}
\label{th+sizeOne}
Let $a, b$ and $c$ be three positive integers, and let $g = \gcd(a+c, b+c)$. The game $(\{a,b\},\{c\})$ is:
\begin{itemize}
\item strongly dominated by  Left  if $g \leq c$,
\item weakly dominated by  Left   with period $(\P^{g-c} \L^{2c - g} \N^{g-c})$ if $c < g <2c$,
\item ultimately impartial with period $(\P^c \N^ c)$ if $g = 2c$ ,
\item weakly dominated by  Right  with period $(\P^c \R^{g - 2c} \N^c)$ if $g > 2c$ .
\end{itemize}

\end{theorem}

\begin{proof}
Throughout this proof we write $n = qg + r$, with $0 \leq r < c$.

We start by proving the following claim which holds in all four cases.
\begin{claim}
 If $(n\bmod g)<c$  then $o_R(n)=L$ for large enough $n$.
\end{claim}

\begin{proof}
After both players play once, the number of tokens decreased by either $a+c$ or $b + c$ depending on which move  Left   played. By the results on the coin problem, we know that if $q$ is large enough, then $qg$ can be written as $\alpha (a + c) + \beta (b +c)$, with $\alpha$ and $\beta$ two non-negative integers. If  Left   is playing second, a strategy can be to play $a$ $\alpha$ times, and $b$ $\beta$ times. After these moves, it is Right's turn to play, and the position is $r < c$. Consequently  Right  now has no move and loses the game.
\end{proof}

We will now use this claim to prove the result in the four different cases.

For the first case, we have $g \leq c$. For any integer $n$, we have $(n \mod g) < g \leq c$. Consequently, by the claim above, there is an integer $n_0$ such that for any $n \geq n_0$,  $o_R(n)=L$. This also implies that for any $n \geq n_0 + a$,  $0_L(n)=L$  since she plays to $n-a>n_0$ and, by the claim, $o_R(n-a)=L$.
Thus the outcome is $\L$ for any position $n$ large enough.

For the three remaining cases, we will show that the following four properties holds when $n$ is large enough. The result of the theorem immediately follows from these four properties.
\begin{enumerate}
\item \label{proof+abc+1} if $r < c$, then  Left   wins playing second,
\item \label{proof+abc+2} if $r \geq g - c$, then  Left   wins playing first,
\item \label{proof+abc+3} if $r \geq c$, then  Right  wins playing first,
\item \label{proof+abc+4} if $r < g-c$, then  Right  wins playing second.
\end{enumerate}
We now prove these four points:
\begin{enumerate}
\item This point is exactly the claim above.
\item If $r \geq g -c$, and $n$ is large enough, then  Left   can play $a$. The position after the move is such that $n-a \equiv r-a \equiv r+c \mod g$. Moreover, since $g-c \leq r < g$, we know that $ g \leq r +c< g+c$. From~\cref{proof+abc+1}, we know that $o_R(n)=L$   if $n-a$ is large enough, so  Left has a winning strategy as a first player if $r \geq g -c$.

\item If $r \geq c$, and  Right  plays first, then whatever  Left   plays, after an even number of moves,  Right  still has a move available. Indeed, let $n'$ be the position reached after an even number of moves. The number of tokens removed, $n -n'$ is a multiple of $g$. Consequently, $n' = (n \mod g)$. Since $(n \mod g) \geq c$, this implies that $n' \geq c$, and  Right  can play $c$. This proves that  Right  will never be blocked, and  Left   will eventually lose the game.

\item Finally, if $r < g-c$, then  Left   playing first can move to a position $n'$ equal to either $n-a$ or $n-b$. Since $a \equiv b \equiv -c \mod g$, in both cases, we have $n' \equiv r +c \mod g$. Since $c \leq r + c < g$, by the argument above, we know that  Right  playing first on $n'$ wins. Consequently,  Left   playing first on $n$ loses.

\end{enumerate}
\end{proof}

When $c>b$ and $b\geq 2a$, which is included in the first case, we know the whole outcome sequence. This will be useful in next Section.

\begin{theorem}
	\label{thm:abc}
	The outcome sequence of the game $(\{a, b\}, \{c\})$, with $c > b$ and $b \geq 2a$ is the following:
	$$ \P^a\L^{c-a}\N^a\L^{\infty}$$
\end{theorem}
\begin{proof}
	We show the result by induction on $n$, the position of the game.
	\begin{itemize}
		\item If $n < a$, then none of the player has a move and thus  $o(n)=\P$.
		\item If $ a \leq n < c$, then only  Left  has a valid move and thus $o(n)=\L$.
		\item If $ c \leq n < a+ c$, then  Right  has a winning move to a position $n-c \leq a$ which has outcome $\P$, and  Left  has a winning move to a position with outcome either $\P$ or $\L$
		thus $o(n)=\L$.
		\item Finally, if $n \geq a +c$, then  Right  has no winning move, and  Left  has at least one winning move. Indeed, since $k \geq a$, we can't have at the same time $n-a$ and $n-a-k$ in the interval $[c, a+c[$. So at least one of $n-a$ and $n-a-k$ is not in this interval, and is either a $\P$-position or a $\L$-position by induction.
	\end{itemize}
\end{proof}




\section{When both sets have size 2}


The goal of this section is to investigate the sequence of outcomes for the game $G = (\sl, \sr)$  with $S_L = \{a, b\}$ and $S_R = \{c, d\}$. In particular, if we suppose that $a$ and $b$ are fixed, we would like to characterize for which positions the game $G$ is eventually $\L$. The picture on Figure~\ref{fig:domination} gives an insight of what is happening. On the figure on the left, we have an example with $b \geq 2a$. In this case, the game $G$ is almost always eventually $\L$, except when the point $(c, d)$ is close to the diagonal, i.e., when $|d-c|$ is close to zero. When $(c,d)$ is close to the diagonal, the behavior seems more complicated, and we won't give a characterization here.

When $b < 2a$, the behavior is more complicated, but shares some similarities with the previous case. From the picture on the right in Figure~\ref{fig:domination} we can see that there are some lines such that if the point $(c,d)$ is far enough from these lines, then the game is eventually $\L$. Again, when the point is close to these lines, the behavior is more complex, and we won't try to characterize it here. In all cases, we can see that if $a$ and $b$ are fixed, for almost all of the choices of $c$ and $d$,  Left  dominates.

In the rest of this section, we will assume that we have $d > c > b$. We start by the case $b \geq 2a$ which is easier to analyse.

\begin{figure}
	\centering
	\includegraphics[width=.45\textwidth]{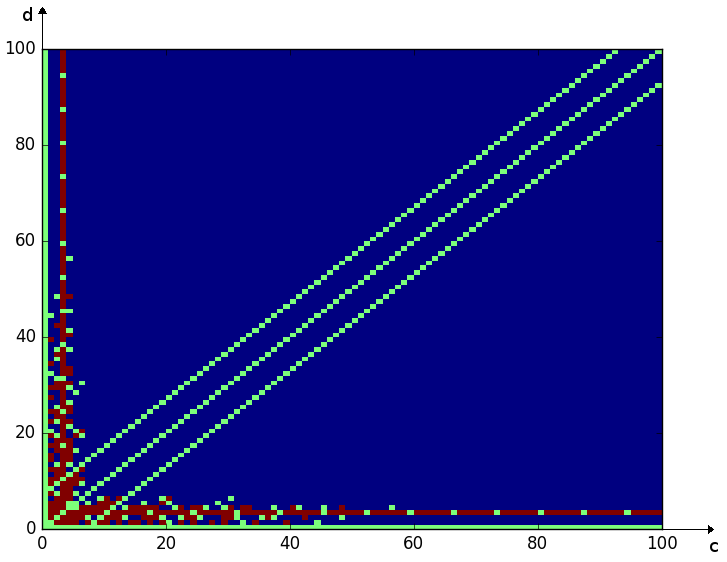}
	\includegraphics[width=.45\textwidth]{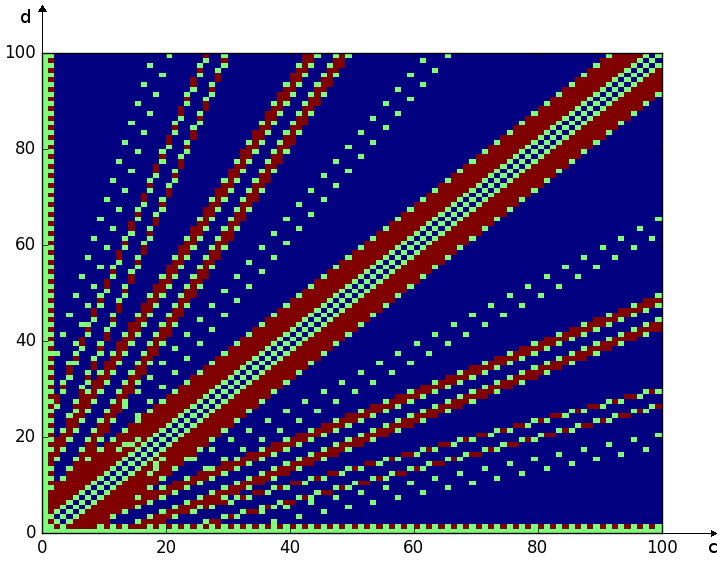}

	{\vspace*{-1.5ex}\tiny $a = 4$ and $k = 7$ \hspace*{15em} $a = 7$ and $k = 2$}
	
	\caption{\label{fig:domination}Properties of the outcome sequences for $G = (\{a, a+k\}, \{c,d\})$. The parameters $a$ and $k$ are fixed, and the pictures are obtained by varying the parameters $c$ and $d$. The point at coordinate $(c,d)$ is blue if the corresponding game is eventually $\L$, red if it is eventually $\R$, and green if there is a mixed period.}
\end{figure}

\subsection{Case $b \geq 2a$}

We start by the case where $b \geq a$, and show that in this case $G$ is ultimately $\L$ if $(c,d)$ is far enough from the diagonal.

\begin{theorem}
	Assume $b \geq 2a$, and $d > c + b$, then $\sl \succ \sr$. More precisely, the outcome sequence is:
	$$ \P^a\L^{c-a}\N^a \L^{d-c -a} \N^a  \L^{\infty}.$$
\end{theorem}
\begin{proof}
	Again, we will show this result by induction on $n$, the starting position of the game. Let $G'$ be the game $(\{a, a+k\}, \{c\})$. If $n < d$, then $G$ played on $n$ has the same outcome as $G'$, since playing $d$ is not a valid move for  Right  in this case. Consequently we can just apply Theorem~\ref{thm:abc}, and get the desired result.
	Otherwise, there are two possible cases:
	\begin{itemize}
		\item If $d \leq n < d + a$, then  Right  has a winning move to the position $n-d < a$, and  Left  has a winning move by playing his strategy for the game $G'$ on $n$. Indeed, this leads to a position $n-x < d$ for some $x \in \{a, a+k\}$ with outcome either $\P$ or $\L$ for $G'$ and consequently also for $G$, since $d$ cannot be played anymore at this point.
		
		\item If $n \geq d +a$, denote by $I_1$ and $I_2$ the two intervals containing the $\N$-position, i.e.,  $I_1 = [c, c+a[$, and $I_2 = [d, d+a[$. Since $k \geq a$, we can't have that $n-a$ and $n-a-k$ are both in $I_1$, or both in $I_2$. Additionally, since $d > c + a + k$, we can't have both $n-a-k \in I_1$ and $n-a \in I_2$ at the same time. Consequently, one of $n-a$ and $n-a-k$ has outcome either $\L$ or $\P$, and  Left  has a winning move on $n$.
	\end{itemize}
\end{proof}

\subsection{General case}

In the general case, we will again prove that if we fix $a$ and $b$, for most choices of $c$ and $d$ the outcome is ultimately $\L$. The exceptional cases are slightly more complicated to characterize. The characterization is related to the following definition:


\begin{definition}
Given an integer $a$, and a real number $\alpha \geq 1$, we denote by $T_{a, \alpha}$ the set of points defined by:
\begin{itemize}
\item $T_{0, \alpha} =  \{ (c,d)\; : \gcd(c, d) \geq \frac {\max(c,d)} \alpha  \}$;
\item for $a \geq 1$, $T_{a, \alpha}$ is obtained from $T_{0, \alpha}$ by a translation of vector $(-a, -a)$.
\end{itemize}
\end{definition}

We can remark that, for any $\alpha$ and $\beta$ with $\beta \geq \alpha$, we have $T_{0, \alpha} \subseteq T_{0, \beta}$.  We now prove some properties of the sets $T_{a, \alpha}$ which will be usefull for the proofs later on.

\begin{lemma}
\label{lem+Tacharact}
Assume that there are some positive integers $x, y, u$ and $v$ such that $xu - yv = 0$ with $(u,v) \neq (0,0)$, then $(x, y) \in T_{0, \max(u,v)}$.
\end{lemma}
\begin{proof}
Up to dividing $u$ and $v$ by $\gcd(u,v)$, we can assume that $u$ and $v$ are coprime. Then, the equation is $xu = yv$. Consequently, $u$ is a divisor of $yv$, and since $u$ and $v$ are coprimes, this means that $u$ is a divisor of $v$. We can write $y = gu$, and consequently we have $xu = yv = vgu$. This means that $x = vg$, and $g = \gcd(x,y)$. Consequently, $\frac {\max(x, y)} {\gcd(x,y)} = \max(u,v)$, and $(x,y) \in T_{0, \max(u,v)}$.
\end{proof}

Given two points $p = (x,y)$ and $p' = (x',y')$, we denote by $\dist(p,p')$ the distance between these two points according to the $1$-norm: $\dist(p, p') = |x -  x'| + |y - y'|$. If $\mathcal D$ is a subset of $\mathbb N^2$, we denote by $\dist(p,\mathcal D) = \min\{\dist(p, p''), p'' \in \mathcal D \}$ the distance of the point $p$ to the set $\mathcal D$.

\begin{lemma}
\label{lem+Tadist}
Assume that there are some positive integers $x, y, u, v$ and $a$ such that $|xu - yv| \leq a$, then $\dist((x, y), T_{0, \max(u,v)}) \leq a(u+v)$.
\end{lemma}
\begin{proof}
Let $r = xu - yv$, with $|r| \leq a$, and $g = \gcd(u,v)$. By definition, $r$ is a multiple of $g$, and we can write $r = qg$ for some integer $q$. Additionally, by Bézout's identity, we know that there exists two integers $u'$ and $v'$ such that $uu' + vv' = g$, and $|u'| \leq u$ and $|v'| \leq v$. Consider the point $(x', y')$, with $x' = x - qu'$, and $y' = y + qv'$. We have the following:
\begin{align*}
    x'u - y'v = xu + yv - q(uu' + vv') = r - q g = 0
\end{align*}
By~\Cref{lem+Tacharact}, we know that $(x', y') \in T_{0, \max(u,v)}$. Additionally, $\dist((x,y), (x', y')) = |qu'| + |qv'| \leq |r| (u+v) \leq a(u+v)$. This proves the Lemma.
\end{proof}

For any $a$ and $\alpha$, the set $T_{a, \alpha}$ satisfies the following properties:
\begin{lemma}
\label{lem+Talines}
For any $a$ and $\alpha$, the set $T_{a, \alpha}$ is the union of a finite set of lines.
\end{lemma}
\begin{proof}
Since $T_{a,\alpha}$ can be obtained from $T_{0,\alpha}$ by a translation, we only need to prove the result in the case $a = 0$. Let $\mathcal D$ be the union of the lines with equation $xu - yv = 0$, for all $u,v \leq \alpha$. The set $\mathcal D$ is the union of a finite number of lines. By~\Cref{lem+Tacharact}, we know that $\mathcal D \subseteq T_{0, \alpha}$. Reciprocally, let $(x, y)$ be a point in $T_{0, \alpha}$, and let $g = \gcd(x, y)$. We can write $x = x'g$, and $y = y'g$ for some integers $x'$ and $y'$. We have the following:
\begin{align*}
    xy' - yx' = x'y'g - y'gx = 0
\end{align*}
Additionally, we have $x' = \frac x g \leq x {\frac {\max(x,y)} \alpha} \leq \alpha$, and similarly for $y'$. Consequently, $(x,y) \in \mathcal D$, and $T_{0, \alpha} = \mathcal D$.
\end{proof}

The goal in the remaining of this section is to prove the following theorem:
\begin{theorem}
\label{th+twoVstwo}
Let $a, b, c$ and $d$ be positive integers, let $A = \lceil \frac a {b-a} \rceil +1$. Assume that $\dist((c, d), T_{a, A}) \geq 2A (a +2b)$, then the partizan subtraction game with $S_L = \{a, b\}$, and $S_R = \{ c,d\}$ is ultimately $\L$.
\end{theorem}

Given two integers $i$ and $j$, we define the following intervals:
\begin{itemize}
	\item $I^\P _{i,j} = [\alpha_{i,j}, \alpha_{i,j} + a - (i+j) (b-a)[$
	\item $I^\N _{i,j} = [\beta_{i,j}, \beta_{i,j} + a - (i+j-1) (b-a)[$
\end{itemize}
where
\begin{itemize}
	\item $\alpha_{i,j} = i (d + b) + j (c + b)$,
	\item  and $\beta_{i,j} = \alpha_{i,j} - b$.
\end{itemize}
 Denote by $I^\P$ the set $\cup_{i,j} I_{i,j} ^\P$, and similarly, $I^\N = \cup_{i,j} I_{i,j} ^\N$. Note that $I^\P _{i,j}$ is empty if $i+j \geq \lceil\frac a k\rceil$, and $I^\N _{i,j}$ is empty if $i+j \geq \lceil\frac a k\rceil + 1$. Our goal is to show that, under the conditions in the statement of the theorem, the set $I^\N$ is the set of $\N$-positions, $I^\P$ the set of $\P$-positions, and all the other positions have outcome $\L$. In particular, since both $I^\P$ and $I^\N$ are finite, this will imply that the outcome sequence is eventually $\L$. Before showing this, we prove that under the conditions of the theorem the intervals $I^\P_{i,j}$ and $I^\N_{i,j}$ satisfy the following properties.

\begin{lemma}
	\label{lem+Iprop}
	Fix the parameters $a$ and $b$, and let $A = \lceil \frac a {b-a}\rceil +1$. Assume that $c$ and $d$ are such that $\dist((c,d), T_{b, A}) \geq 2 A (a +2b)$, then the intervals $I_{i,j}^\N$ and $I^\P_{i,j}$ satisfy the following properties:
	\begin{enumerate}[(i)]
		\item \label{I+disj} they are pairwise disjoint,
		\item \label{I+prec} there is no interval $I^\P_{i', j'}$ or $I^\N_{i', j'}$ intersecting any of the $b$ positions preceding $I_{i,j}^\N$,
		\item \label{I+plusc} $I_{i,j}^\P + c = I_{i, j+1} ^ \N$,
		\item \label{I+plusd} $I_{i,j}^\P + d = I_{i+1, j} ^ \N$,
		\item \label{I+int} $(I_{i,j}^\N + a) \cap (I_{i,j}^\N + b) = I_{i,j}^\P$.
	\end{enumerate}
\end{lemma}
\begin{proof}
    The points $(\ref{I+plusc})$, $(\ref{I+plusd})$ and $(\ref{I+int})$ are just consequences of the definitions of $I^\P_{i,j}$ and $I^\N_{i,j}$. Consequently, we only need to prove the two other points.

    We know that $I^\N_{i,j}$ and $I^\P_{i,j}$ are empty when $i+j \geq \lceil \frac a {b-a}\rceil +1 = A$, consequently, we will assume in all the following that the indices $i,j,i'$ and $j'$ are all upper bounded by $A$.
    We first show the following claim. The rest of the proof will simply consists in applying this claim several times.

    \begin{claim}
        \label{claim+distance}
        Assume that there is an integers $B$, and indices $i,j, i', j'\leq A$, such that one of the following holds:
        \begin{itemize}
            \item $|\alpha_{i,j} - \alpha_{i', j'}| \leq B$
            \item $|\beta_{i,j} - \beta_{i', j'}| \leq B$
            \item $|\alpha_{i,j} - \beta_{i', j'}| \leq B$
        \end{itemize}
        Then in all three cases we have $\dist((c,d), T_{b, A}) \leq 2 A (B+b)$.
    \end{claim}
    \begin{proof}
        The first two cases are equivalent to the inequality $|(i - i')(d+b) + (j - j')(c+b)| \leq B$, and the result follows by applying~\Cref{lem+Tadist}. The third case is equivalent to $|(i - i')(d+b) + (j - j')(c+b) + b| \leq B$. Using the triangle inequality, this implies $|(i - i')(d+b) + (j - j')(c+b)| \leq B + b$, and the result follows from~\Cref{lem+Tadist}.
    \end{proof}

    We will prove the points~$(\ref{I+disj})$ and~$(\ref{I+prec})$ by proving their contrapositives. In other words, assuming that one of these two conditions does not hold, we want to show that $\dist((c,d), T_{b, \alpha}) \leq 2\alpha(a+b)$.

    We first consider the point $(\ref{I+disj})$. First, assume that there are two intervals $I^\P_{i,j}$ and $I^\P_{i', j'}$ such that the two intervals intersect. Then, the  Left  endpoint of one of these two intervals is contained in the other interval. Without loss of generality, we can assume that $\alpha_{i,j} \in I^\P_{i', j'}$. This implies:
	\begin{align*}
	    \alpha_{i',j'} \leq \alpha_{i,j} \leq \alpha_{i',j'} + a - (b-a)(i'+j') \\
	    0 \leq \alpha_{i,j} - \alpha_{i', j'} \leq a - (b-a) (i' + j') \leq a \\
	\end{align*}
	By~\cref{claim+distance}, this implies $\dist((c,d), T_{b,A}) \leq 2A(a+b)$.

    Similarly, if we assume that $I^\N_{i,j}$ and $I^\N_{i', j'}$ intersect, then this implies without loss of generality that $\beta_{i,j} \in I^\N_{i', j'}$, and consequently, $0 \leq \beta_{i,j} - \beta_{i', j'} \leq a - (i+j-1)k \leq a$. Again, using~\Cref{claim+distance}, this implies $\dist((c,d), T_{b, A}) \leq 2(a+b)A$.

    Finally, if $I^\N_{i',j'}$ and $I^\P_{i,j}$ intersect, then either $0 \leq \alpha_{i,j} - \beta_{i', j'} \leq a$ if $\alpha_{i,j} \in I^\N_{i',j'}$ or $0 \leq \beta_{i', j'} - \alpha_{i,j} \leq a$ if $\beta_{i', j'} \in I^\P_{i,j}$. In both cases, the~\cref{claim+distance} gives the desired result.

    The proof for the point~$(\ref{I+prec})$ is essentially the same as above. If $I^\N_{i',j'}$ intersects one of the $b$ positions preceding $I^\N_{i,j}$, then we have the two inequalities:
    \begin{align*}
        \beta_{i', j'} + a-(i' + j' -1) k &\geq \beta_{i,j} - b & \beta_{i', j'} \leq \beta_{i,j}
    \end{align*}
    From these inequalities we can immediately deduce $ -a-b \leq \beta_{i', j'} - \beta_{i,j} \leq 0$. The inequality $\dist((c,d), T_{a+k, A}) \leq 2A(3a+2k)$ follows immediately from~\Cref{claim+distance}. Similarly, if the interval $I^\P_{i',j'}$ intersects one of the $b$ positions preceding $I^\N_{i,j}$, then we have the two inequalities:

    \begin{align*}
        \alpha_{i', j'} + a-(i' + j') (b-a) &\geq \beta_{i,j} - b & \alpha_{i', j'} \leq \beta_{i,j}
    \end{align*}
    This implies $ -(a+b) \leq \alpha_{i', j'} - \beta_{i,j} \leq 0$, and again the result holds by~\cref{claim+distance}.
\end{proof}

We now have all the tools needed to prove the theorem.

\begin{proof}{of~\Cref{th+twoVstwo}}

    Let $a,b,c,d$ be integers, and let $A = \lceil \frac a {b-a}\rceil$, and assume that $\dist((c,d), T_{b, A}) \geq 2A(a+2b)$. We know that the four properties of~\Cref{lem+Iprop} hold. We will show by induction on $n$ that for any position $n \geq 0$, if $n \in I^\P$, then $n$ is a $\P$-position, if $n \in I^\N$, then it is a $\N$-position, and otherwise it is an $\L$-position. The inductive case is treated in the same way as the base case.

    First, assume that $n \in I_{i,j}^\N$ for some indices $i$ and $j$ such that $i+j \geq 1$.  Left  has a winning move by playing $a$. Indeed, the interval $I_{i,j}^\N$ has length at most $a$, and using the condition~$(\ref{I+prec})$ from \Cref{lem+Iprop} and the induction hypothesis, $n-a$ is a $\L$-position. If $i > 0$, then  Right  playing $c$ leads to the position $n-c \in I_{i-1, j}^\P$ by condition~$(\ref{I+plusc})$. This position is a $\P$-position using the induction hypothesis. If $j > 0$, then similarly,  Right  can play $d$, and put the game in the position $n-d \in I_{i, j-1}^\P$ by condition~$(\ref{I+plusd})$. This position is a $\P$-position using the induction hypothesis.

	Suppose now that $n \in I_{i,j}^\P$. If $i$ and $j$ are both zero, then none of the players have any move, and $n$ is a $\P$-position. Otherwise, if  Left  plays either $a$ or $b$, this leads to a position $n' \in I_{i,j}^\N$ by condition~$(\ref{I+int})$. Using the induction hypothesis, $n'$ is an $\N$-position, and  Left  has no winning move. Right's only possible winning move would be to a $\P$-position $n'$. Using the induction hypothesis this means $n' \in I^\P$. However, this would mean by conditions~$(\ref{I+plusc})$ and~$(\ref{I+plusd})$ that $n \in I^\N$, which is a contradiction of the property~$(\ref{I+disj})$ that $I^\N$ and $I^\P$ are disjoint. Consequently,  Right  has no winning move.
	
	Finally, suppose that $n \not \in I^\P \cup I^\N$. We will show that  Left  has a winning move on $n$, and  Right  does not. Since $I_{0,0}^\P = [0, a[$, we can assume $n \geq a$, and  Left  can play $a$. Suppose that  Left's move to $n-a$ is not a winning move, and let us show that  Left  has a winning move to $n-a-k$. Since  Left's move to $n-a$ is not a winning move, this means that $n-a \in I_{i,j}^\N$ for some integer $i,j$ with $i+j \geq 1$. Consequently we have $n\geq b$, and playing $b$ is a valid move for  Left. By condition~(\ref{I+int}), we can't have $n-b \in I_{i,j}^\N$ since otherwise we would have $n \in I_{i,j}^\P$. Moreover, we can't have either $n-b \in I^\N_{i', j'}$ for some $(i', j') \neq (i,j)$ since it would contradict condition~(\ref{I+prec}). Consequently, $n-b \in I^\L$, and using the induction hypothesis, this is a winning move for  Left.
	The only possible winning move for  Right  would be to play to a position $n'$ which is a $\P$-position. Using the induction hypothesis, this means that $n' \in I^\P$. However using the conditions~$(\ref{I+plusc})$ and~$(\ref{I+plusd})$ this would also imply $n \in I^\N$, a contradiction.	
\end{proof}

\begin{corollary}
	Under the conditions of the theorem, the game $G$ is ultimately $\L$.
\end{corollary}
\begin{proof}
	Since $I^\N_{i,j}$ and $I^\P_{i,j}$ are both empty if $i + j > a$, the two sets $I^\L$ and $I^\N$ are finite, and the result follows from the theorem.
\end{proof}


\end{document}